\documentclass[11pt]{amsart}
\usepackage{amsmath,amsfonts,amsthm,amssymb,amscd}
\def\classification#1{\def\@class{#1}}
\classification{\null}

\DeclareFontFamily{OT1}{rsfs}{}
\DeclareFontShape{OT1}{rsfs}{n}{it}{<-> rsfs10}{}
\DeclareMathAlphabet{\mathscr}{OT1}{rsfs}{n}{it}

\newcommand{\R}{{\mathbb R}}

\newcommand{\F}{\mathbb{F}}

\newtheorem{theorem}{Theorem}
\newtheorem{lemma}[theorem]{Lemma}
\newtheorem{corollary}[theorem]{Corollary}
\theoremstyle{remark}
\newtheorem{remark}[theorem]{Remark}

\thanks{The authors were partially supported by the NSF Grant DMS-1045404}


\title{Areas of triangles and Beck's theorem in planes over finite fields}
\author{Alex Iosevich}
\address{Alex Iosevich: Department of Mathematics, University of Rochester, Rochester, NY}
\email{iosevich@math.rochester.edu}

\author{Misha Rudnev}
\address{Misha Rudnev: Department of Mathematics, University of Bristol, Bristol BS8 1TW, UK}
\email{m.rudnev@bristol.ac.uk}

\author{Yujia Zhai}
\address{Yujia Zhai: Department of Mathematics, University of Rochester, Rochester, NY}
\email{yzhai@u.rochester.edu}

\begin{document}

\begin{abstract} It is shown that any subset $E$ of a plane over a finite field $\F_q$, of cardinality $|E|>q$ determines not less than $\frac{q-1}{2}$ distinct areas of triangles, moreover once can find such triangles sharing a common base.

It is also shown that if $|E|\geq 64q\log_2 q$, then there are more than $\frac{q}{2}$ distinct areas of triangles sharing a common vertex. The result follows from a finite field version of the Beck theorem for large subsets of $\F_q^2$ that we prove. If $|E|\geq 64q\log_2 q$, there exists a point $z\in E$, such that there are at least $\frac{q}{4}$ straight lines incident to $z$, each supporting the number of points of $E$ other than $z$ in the interval between $\frac{|E|}{2q}$ and $\frac{2|E|}{q}.$ This is proved by combining combinatorial and Fourier analytic techniques.

We also discuss higher-dimensional implications of these results in light of recent developments.

\end{abstract}

\maketitle

\section{Introduction}

\vskip.125in

The basic question of Erd\H os combinatorics is to determine whether a sufficiently large discrete set determines a configuration of a given type. See, for example, \cite{AP95}, \cite{BMP}, \cite{TV10} and the references contained therein for a description of related problems and their consequences. Perhaps the most celebrated of these is Szemer\'edi's theorem's which says that a subset of the integers of positive density contains an arithmetic progression of any given length.

In the metric plane, interesting geometric configurations are congruent line segments or triangles.
In this context, Guth and Katz (\cite{GK}) recently resolved the long-standing Erd\H os distance conjecture (\cite{E}) in ${\Bbb R}^2$ by proving that a planar set $E$  of $n$ points determines $\Omega(\frac{n}{\log n})$ distinct distances, or, equivalently, distinct congruence classes of line segments.

Another conjecture of Erd\H os, Purdy, and Strauss (\cite{EPS}) suggested that a non-collinear set of $n$ points in $\R^2$ determines at least $\lfloor\frac{n-1}{2}\rfloor$ distinct nonzero triangle areas, the most economical configuration, for an even $n$, being two parallel copies of evenly distributed $\frac{n}{2}$ points lying on a line.
It was resolved in 2008 by Pinchasi (\cite{Pi}), yet a linear lower bound was known long before then. Pinchasi also showed that the triangles yielding distinct areas may be chosen to share a common base.

Roche-Newton with the first two listed authors of this paper (\cite{IRR}) proved that a non-collinear $n$-point set $E\subset \R^2$ generates $\Omega(\frac{n}{\log n})$ distinct triangle areas for triangles $OAB$, where $A,B\in E$ and $O$ is {\em any} origin.

Similar questions have been asked in the higher dimension Euclidean space, and many are still open. See, for example, (\cite{BMP}) and the references contained therein.

\vskip.125in

Geometric combinatorics in vector spaces over finite fields has also received much recent attention. Geometric problems in vector spaces over a finite field $\F_q$ can often be effectively resolved using Fourier analysis, provided that the point sets in question are sufficiently large. See e.g. \cite{G}, \cite{HI08}, \cite{V}, \cite{HIKR11}, and the references contained therein.

Fourier techniques have not proven to be particularly efficient for sufficiently small subsets of finite fields, where the fundamental tools come from arithmetic combinatorics, as is demonstrated in a pioneering work of Bourgain, Katz, and Tao (\cite{BKT}). Unfortunately, quantitative results involving small sets are still far behind their Euclidean prototypes. For instance, the celebrated Beck's theorem (\cite{Be}), whose finite field version, Theorem \ref{beckfinitefield} below, is developed in this paper, states that a set of $n$ points in the Euclidean plane, with no more than $cn$ collinear points, with some absolute $c$, determines $\Omega(n^2)$ distinct lines drawn through pairs of distinct points. The best currently known exponents, for small sets, are due to Helfgott and the second listed author (\cite{HR}), who proved that for any $(E=A\times A)\subset \F_q^2,$ where $q$ is a prime, and $|A|<\sqrt{q}$, \footnote{Throughout the paper we use the $|\cdot|$ notation to denote the cardinality of a finite set.
} the number of distinct lines is $\Omega(n^{1 + \frac{1}{267}})$. The power $\frac{1}{267}$, added to $1$ in the latter estimate has since been improved a few times, implicit in the recent work of Jones (\cite{J}), but would still remain very far from $1$ as it is in Beck's theorem.
\vskip.125in
In this paper we study the distribution of areas of triangles determined by subsets of the finite plane ${\mathbb F}_q^2$ and draw some conclusions about
volumes determined by subsets of ${\mathbb F}_q^d$. Here ${\mathbb F}_q$ is the field with $q$ elements and ${\mathbb F}_q^d$ is the $d$-dimensional vector space over this field. More precisely, let $(x^1, \ldots, x^{d+1})$ denote a $(d+1)$-tuple of vectors from ${\mathbb F}_q^d$. Given a set $E\subseteq \F^d_q$, define
\begin{equation} V_d(E)=\left\{\det(x^1-x^{d+1}, \ldots, x^d-x^{d+1}): x^j \in E \right\} \setminus\{0\}\label{dvol}\end{equation} as the set of $d$-dimensional nonzero volumes, defined by $(d+1)$-simplices whose vertices are in $E$,  as well as for some fixed $z\in E$, the set of pinned nonzero volumes
\begin{equation} V^z_d(E)=\left\{\det(x^1-z, \ldots, x^d-z): x^j \in E \right\} \setminus\{0\}.\label{dvolp}\end{equation}
Note that if $x^j = (x^j_1,\ldots,x^j_d)$, the subscripts referring to the coordinates relative to the standard basis in $\F_q^d$, an element of $V_d(E)$, generated by the $(d+1)$-tuple $(x^1, \ldots, x^{d+1})$ equals the determinant of the following $d+1$ by $d+1$ matrix:
\begin{equation}
V_d(x^1, \ldots, x^{d+1})=\left|\begin{array}{ccccc} 1&\ldots&1\\
x_1^1&\ldots&x^{d+1}_1\\
\vdots&\ddots&\vdots\\
x^1_{d}&\ldots&x^{d+1}_{d}
\end{array}
\right|.
\label{dvolpp}\end{equation}

\subsection{Statement of results:}

The main result of this note is the following theorem on areas determined by finite subsets of two-dimensional vector spaces over finite fields.

\begin{theorem} \label{2d} Let $E \subset {\mathbb F}_q^2$. Then the following hold.

\vskip.125in

i) Suppose that $|E|>q$. Then $|V_2(E)| \geq \frac{q-1}{2}$, and the triangles giving at least

$ \frac{q-1}{2}$ distinct areas can be chosen such that they share the same base.

\vskip.125in

ii) Suppose that $|E| \ge 64q \log_2(q)$, and $q\geq C$ for some absolute $C$. Then there

exists $z \in E$ such that

$$|V^z_2(E)| > \frac{q}{2}.$$
\end{theorem}
\begin{remark}\label{offbytwo}
The claim i) can be viewed as a slightly weaker finite field version of the aforementioned result of Pinchasi in $\R^2$. Namely, we are off by a constant in the sense that we do not know an example of a set with $q+1$ points in $\F_q^2$, which would not generate {\em} all possible $q-1$ nonzero areas of triangles. The question of what is the minimum size of $E\subset \F_2^q$ to yield all possible areas in $\F_q\setminus\{0\}$ is open. It seems reasonable to conjecture that $|E|=q+1$ would be necessary and sufficient, for any $q$. The analog of the conjecture in higher dimensions would be that $|E|=q^{d-1}+1$, necessary and sufficient to yield all possible $d$-dimensional volumes. In this direction, Corollary \ref{dveruki} below shows that if $|E|\geq 2q^{d-1}$, then all the possible volumes are determined. \end{remark}

\begin{remark} In the context of recent work on sufficiently large sets in $\F_q^2$, Theorem \ref{2d} is an improvement over the earlier results \cite{HI08}, \cite{V1}, which established the threshold $|E|=\Omega(q^{\frac{3}{2}})$ in order for a general $E\subseteq \F_q^2$ to determine $\Omega(q)$ distinct areas of triangles. \end{remark}

\vskip.125in

The claim ii) of Theorem \ref{2d} follows from the following theorem, which can be regarded as a finite field version of the aforementioned theorem due to Beck (\cite{Be}), for sufficiently large point sets.

\begin{theorem} \label{beckfinitefield} (Beck's theorem in ${\mathbb F}_q^2$). Suppose that $E \subset {\mathbb F}_q^2$ with $q\geq C$, for some absolute $C$, and
$$|E| \geq 64 q \log q.$$ Then pairs of distinct points of $E$ generate at least $\frac{q^2}{8}$ distinct straight lines in $\F_q^2$. Moreover, there exists a point $z\in E$ and at least $\frac{q}{4}$ straight lines incident to $z$, each supporting at least $\frac{1}{2}\frac{|E|}{q}$ and fewer than
$2\frac{|E|}{q}$ points of $E$, other than $z$.
\end{theorem}

\begin{remark} It is well known that Fourier analysis yields nearly optimal estimates over finite fields for sufficiently large sets. For instance, Garaev (\cite{G}) proves an optimal sum-product lower bound  for $|A+A|\,+\,|A\cdot A|$, when $A\subseteq \F_q$ is such that $|A|>q^{\frac{2}{3}}$. The first quantitative estimate in this direction was proved by Hart, the first listed author of this paper and Solymosi in \cite{HIS07}.
Similarly, in $\F_q^2$ the ``threshold'' for what can be regarded as a sufficiently large set $E=A\times A$ is usually $q^{\frac{4}{3}}$, and $q^{\frac{3}{2}}$ for general $E$. See, e.g., \cite{HI08}, \cite{V1}, \cite{HIKR11}, and the references contained therein. Theorem \ref{beckfinitefield}, however, delivers an optimal (up to a constant) estimate for $|A|=\Omega \left( \sqrt{q\log_2 q}\right)$. Note that the aforementioned estimate $\Omega(n^{1 + \frac{1}{267}})$ in \cite{HR} for the number of lines generated by $A\times A$ with $|A|=O(\sqrt{q})$ (valid only when $q$ is a prime) is strikingly weaker. In the same vein, sum-product results of \cite{G}, \cite{HIS07}, valid for sufficiently large sets are considerably stronger than what is known for small sets, where the best result so far in prime fields is due to the second listed author (\cite{Ru}), generalised to $\F_q$ by Li and Roche-Newton (\cite{LRN}).
\end{remark}

\begin{remark} The construction in Corollary 2.4 in \cite{HIKR11} implies that if $|E|=o(q^{\frac{3}{2}})$, then there exists $z \in E$ such that $|V^z_2(E)|=o(q)$. In particular, this implies that part ii) of Theorem \ref{2d} cannot be strengthened to say that one gets a positive proportion of the areas from {\em any} fixed vertex, which would be somewhat analogous to the above-mentioned Euclidean result of \cite{IRR}. We do not know whether the logarithmic term in the assumption for part ii) is necessary. It is definitely needed for our proof. \end{remark}

\begin{remark} The forthcoming proof of Theorem \ref{beckfinitefield} is based on Vinh's (\cite{V}) finite field variant, quoted as Theorem \ref{vinhkrutoi} below, of the classical Szemer\'edi-Trotter (\cite{ST}) theorem on the number of incidences $I(E,L)$ of points in $E$ and lines in $L$. Vinh's theorem becomes its equal in the strength of exponents only if the underlying sets $E,L$ involved are rather large, that is if one takes $|E||L|\sim q^{3}$, which amounts to $|E|=\Omega(q^{\frac{3}{2}})$ in the interesting case when $|E|\sim|L|$. Indeed, this is the threshold when the first term in Vinh's incidence estimate (\ref{vest}) dominates, giving  $I(E,L)=O((|E||L|)^{\frac{2}{3}})$ for the number of incidences, as it is the case in the principal term of the celebrated Euclidean Szemer\'edi-Trotter estimate. Beck's theorem, however, does not require the full strength of the Szemeredi-Trotter incidence theorem and would already follow if the fact $k^3$ in the denominator of the first term of (\ref{stb}) below is replaced by $k^{2+\epsilon}$ for some $\epsilon>0$.

In the original paper (\cite{Be}), Beck had $\epsilon = \frac{1}{20}$, rather than $\epsilon = 1$, provided by the Szemer\'edi-Trotter theorem as in the estimate (\ref{stb}) below. In other words, Beck's theorem is weaker than the Szemer\'edi-Trotter theorem. This is precisely the reason why we can afford to use (\ref{vest}) and succeed in obtaining a much better threshold $|E|=\Omega(q\log_2 q)$ in Theorem \ref{beckfinitefield} (rather than $|E|=\Omega(q^{\frac{3}{2}})$) getting a nearly sharp (up to the endpoint term $\log_2 q$) variant of the Beck theorem in the finite plane ${\mathbb F}_q^2$, as to the minimum size of a set $E$ to generate
$\Omega(q)$ distinct straight lines.
\end{remark}

Theorem \ref{2d} can be easily boot-strapped to higher dimensions, since if a set determines a certain number of $(d-1)$-dimensional volumes when restricted to a $(d-1)$-dimensional hyperplane, and on top of that contains at least one point outside of this hyperplane, then it automatically determines at least that many $d$-dimensional volumes. As a consequence of our method, we obtain the following improvement of a result of Vinh (\cite{V2}) who proved that if $|E|\geq (d-1)q^{d-1}$, $d \ge 3$, then $V_d(E)={\mathbb F}_q \ \backslash \{0\}$.

\begin{corollary} \label{dveruki}  Let $E \subset {\mathbb F}_q^d$, $d \ge 3$.

\vskip.125in

i) Suppose that $|E|>q^{d-1}$. Then $|V_d(E)|\geq \frac{q-1}{2}$.

\vskip.125in

ii) Suppose that $|E|\geq 2q^{d-1}$. Then $V_d(E)={\Bbb F}_q \ \backslash \{0\}$.
\end{corollary}

\vskip.25in

\section{Proof of part i) of Theorem \ref{2d}}

\vskip.125in

\begin{proof} The core of the forthcoming proof of claim i) follows the lines of  Lemmas 2.1 and 2.2 in \cite{GR}, which in turn go back to the ``statement about generic projection'' in  \cite{BKT}, Lemma 2.1.

Let us first show that any set $E\subseteq \F_q^2,$ with $|E|>q$,  determines all possible directions. More precisely, every linear subspace $L\subset \F_q^2$ contains a nonzero element of $E-E$. This is a finite field analogue of the well-known result of Ungar (\cite{U}) that $2N$ non-collinear points in the Euclidean plane determine at least $2N$ distinct directions.

Let $L$ be a one-dimensional linear subspace of $\F^2_q$. Consider the sum set
$$S=E+L=\{s=e+l:\,e\in E,l\in L\}.$$ Since $|L||E|>|\F_q^2|=q^2$, there is an element $s\in S$ with more than one representation as a sum. More precisely,
\begin{equation}
s = e_1+l_1 = e_2+l_2,\qquad (e_1,e_2,l_1,l_2)\in E\times E\times L\times L, \qquad l_1\neq l_2.
\label{subsp}\end{equation}
Hence \begin{equation}\label{diff}
l_2-l_1 = e_1-e_2,\end{equation}
which implies that $E$ determines all possible directions in the sense described above.

\vskip.125in

Now average the number of solutions of the equation (\ref{diff}), with $(e_1,e_2,l_1,l_2)\in E\times E\times L\times L$, over the $\frac{q^2-1}{q-1}=q+1$ subspaces $L$. For each $L$, the pair $(e_1, e_2)$, $e_1\neq e_2$ in (\ref{diff}) determines the subspace $L$. Moreover, each $l\in L$ can be represented as a difference $l_2-l_1$ in exactly $q$ different ways. Including the trivial solutions where $e_1=e_2$, we have
$$ |\{(e_1,e_2,l_1,l_2)\in E\times E\times L\times L: \mbox{ (\ref{diff}) holds for some $L$}\}|$$
$$\begin{aligned} &=|E|q(q+1) + |E|^2q \\ & \leq 2q|E|^2.\end{aligned}$$

It follows that there exists a subspace $L$, such that
\begin{equation} \label{pigeonline}|\{(e_1,e_2,l_1,l_2)\in E\times E\times L\times L\}:\mbox{ (\ref{diff}) holds}\} | \leq 2 |E|^2\frac{q}{q+1}. \end{equation}

It follows, by the Cauchy-Schwartz inequality, that with this particular $L$,

\vskip.125in

\begin{equation} \label{popularity}\begin{aligned}
|E+L| & \geq  \frac{|E|^2|L|^2}{|\{(e_1,e_2,l_1,l_2)\in E\times E\times L\times L\}:\mbox{ (\ref{diff}) holds}\}| } \\ &\geq  \frac{q(q+1)}{2}.\end{aligned} \end{equation}

Indeed, if $s\in E+L$, and $\nu(s)$ is the number of realisations of $s$ as a sum, then
$$
\sum_{s} \nu(s) = |E||L|,\qquad \sum_{s} \nu^2(s) =|\{(e_1,e_2,l_1,l_2)\in E\times E\times L\times L\}:\mbox{ (\ref{diff}) holds}\}|,
$$
and by the Cauchy-Schwartz inequality
$$ |S=E+L|\leq \frac{\left(\sum_{s\in S} \nu(s) \right)^2}{\sum_{s\in S} \nu^2(s)}.
$$

\medskip

We conclude from (\ref{popularity}) that there are points of $E$ in at least $\frac{q+1}{2}$ different parallel lines. Moreover one of these lines, as has been shown in the beginning of the proof, has at least two distinct points $e_1,e_2\in E$. It follows that $E$ determines at least $\frac{q-1}{2}$ distinct triangles, with the same base $e_1e_2$ and distinct nonzero heights, that is the third vertex of the triangle lying on different lines parallel to the base. Thus, there are at least $\frac{q-1}{2}$ distinct nonzero triangle areas.

\end{proof}

\section{Proof of part ii) of Theorem \ref{2d}}

\vskip.125in

\subsection{Fourier mechanism}
\label{fouriersubsection}

We shall need the following Fourier-analytic result, which is an easy variant of the corresponding estimate from \cite{HI08} and \cite{HIKR11}.

Let $\chi$ be a non-trivial additive character over $\F_q$ and let $F(x)$ for the characteristic function of a set $F\subseteq \F_q^2$. Define the Fourier transform $\widehat F$ of $F(x)$  as
$$\widehat{F}(\xi) = \frac{1}{q^2} \sum_x F(x) \chi(-\xi\cdot x).$$

\begin{theorem} \label{L2} Let $F,G \subset \F_q^2$. Suppose $0\not \in F$. Let, for $t\in \F_q$,
$$ \nu(t)=|\{(x,y) \in F \times G: x \cdot y=t \}|,$$ where $x \cdot y=x_1y_1+x_2y_2$.

Then
\begin{equation} \label{L2estimate} \sum_t \nu^2(t) \leq {|F|}^2 {|G|}^2 q^{-1}+q |F| |G| \cdot \max_{x\neq 0}|F\cap l_x|, \end{equation} where
$$l_x=\left\{sx: s \in {\mathbb F}_q \setminus\{0\} \right\},$$ with $x \in {\mathbb F}_q^2\setminus\{0\}$.
\end{theorem}

\begin{proof}
To prove the theorem, observe that for any $t\in \F_q$, one has, by the Cauchy-Schwartz inequality:
$$ \begin{aligned} \nu^2(t)&={\left( \sum_{x \cdot y=t} F(x)G(y) \right)}^2 \\ & \leq |F| \; \cdot \sum_{x \cdot y=x \cdot y'=t} F(x)F(y)F(y').\end{aligned}$$

Using the identity
$$
\sum_{x \cdot y=x \cdot y'} 1 = q^{-1} \sum_{x,y\in \F^2_q} \chi(x\cdot y - x\cdot y'),
$$
one has
$$ \begin{aligned}\sum_t \nu^2(t) &\leq |F| \cdot \sum_{x \cdot y=x \cdot y'} F(x)G(y)G(y')
\\ & ={|F|}^2{|G|}^2q^{-1}+q^{-1}|F| \sum_{s \not=0} \sum_{x,y,y'} \chi(sx \cdot (y-y')) F(x)G(y)G(y')
\\ & ={|F|}^2{|G|}^2q^{-1}+q^{3} |F| \sum_{s \not=0} \sum_x {|\widehat{G}(sx)|}^2 F(x)
\\ & ={|F|}^2{|G|}^2q^{-1}+q^{3} |F|\sum_{s \not=0} \sum_{x \not = 0}{|\widehat{G}(x)|}^2 F(sx)
\\ & \leq {|F|}^2{|G|}^2q^{-1}+q |F| |G| \cdot \max_{x\neq 0}|F\cap l_x|,\end{aligned}$$
using the Plancherel identity as well as the fact the assumption that $0\not \in F$. This completes the proof of Theorem \ref{L2}.\end{proof}

\vskip.125in

\subsection{Finite field variant of the Beck theorem}
\label{becksubsection}
 Here we prove Theorem \ref{beckfinitefield}, a variant of the Euclidean theorem due to Beck (\cite{Be}) in the finite field context. Recall that Beck's theorem says that either a positive proportion of a set of $n$ points in the Euclidean plane lie on single line, or there exists a constant multiple of $n^2$ distinct lines each containing at least two points of $E$. Note that Beck's theorem follows easily from the following formulation of the celebrated Szemer\'edi-Trotter incidence theorem (\cite{ST}).

\begin{theorem} \label{ST} Let $E$ be a collection of points in $\R^2$, and $L(E)$ the set of lines determined by distinct pairs of points of $E$. For $k\geq 2$, let $L_k\subseteq L(E)$ denote the lines supporting at least $k$ points of $E$. Then there exists $C>0$ such that
 \begin{equation}\label{stb}
 |L_k| \leq C \left( \frac{|E|^2}{k^3} + \frac{|E|}{k} \right).
\end{equation}
\end{theorem}


We shall need the following finite field variant of the Szemer\'edi-Trotter theorem due to Vinh (\cite{V}).
\begin{theorem} \label{vinhkrutoi} Let $E$ be a collection of points and $L$ a collection of lines in ${\mathbb F}_q^2$. Then
\begin{equation}\label{vest} I(E,L)=\{(e,l) \in E \times L: e \in l \}| \leq |E| \cdot |L| \cdot q^{-1}+\sqrt{q \cdot |E| \cdot |L|}.\end{equation}
\end{theorem}

\vskip.125in

Using $L_k$ in place of $L$ in Theorem \ref{vinhkrutoi}, we see that $I(E,L_k) \geq k|L_k|$. It then follows directly from (\ref{vest}) that if $k>\frac{|E|}{q}$, then
\begin{equation}\label{stb1} |L_k|\leq  \frac{q|E|}{(k - q^{-1}|E|)^2}.
\end{equation}

This leads us to the proof of Theorem \ref{beckfinitefield}.

\begin{proof}
Here $q$ is treated as asymptotic parameter, to which the $o(1)$ notation relates.
In order to avoid messy notation, let us assume, for convenience, that the quantities $q^{-1}|E|$, as well as $\log_2 q$ are integers. The reader shall see that we have more than enough flexibility with the constants to make this work. The key to our proof is the following assertion.

\begin{lemma} \label{alabeck} At least $\frac{|E|^2}{4}$ unordered  pairs of distinct points of $E$ are supported on the subset  ${\mathcal L}\subseteq L(E)$, defined as the set of all lines in $L(E)$ containing between $1+\frac{|E|}{2q}$ and $2\frac{|E|}{q}$ points of $E$. \end{lemma}

\begin{proof}
With a slight abuse of the notation, let $L^j$ be the set of lines from $L(E)$, supporting no fewer than $q^{-1}|E|\cdot 2^j$ and no more than $q^{-1}|E|\cdot 2^{j+1}$ points of $E$, where $j$ is an integer ranging between $1$ and $\log_2 q$. The maximum number of points on a line from $L^j$ is $q^{-1}|E|\cdot 2^{j+1}$, so the number of unordered  pairs of distinct points is smaller than
$q^{-2}|E|^2\cdot 2^{2j+1}$. On the other hand, by (\ref{stb1}),
$$
|L^j|\leq \frac{q|E|} { q^{-2}|E|^2 \cdot (2^j-1)^2 }
$$
Hence, the number  of  unordered  pairs of distinct points of $E$ supported on any $L^j$, for $j\geq 2$ is bounded by
$$
q|E|\frac{2^{2j+1}}{ (2^j-1)^2 } \leq 4|E| q,\mbox{ for }j\geq 2,
$$
and by $8|E|q$ for $j=1$. Then the number  of  unordered  pairs of distinct points of $E$ supported on the union of $L^j$, for $j\geq 1$ is bounded by $4|E|q(1+\log_2 q)$. Then, if $|E| \geq 64 q \log_2 q,$ and $q$ is large enough, the latter bound constitutes only a small proportion of the total number $\frac{|E|(|E|-1)}{2}$ of unordered  pairs of distinct points of $E$. More precisely, it follows, with the above choice of constants (for $q$ large enough and $q=o(|E|)$) that the number of unordered  pairs of distinct points of $E$, supported on lines in $L(E)$, each supporting at most $2q^{-1}|E|$ points, is at least  $\frac{5|E|^2}{12}.$

There are only $q(q+1)$ distinct lines in $\F_q^2$, so lines from $L(E)$ with fewer than $1+\frac{|E|}{2q}$ points (for $q$ large enough and $q=o(|E|)$) may support fewer than $\frac{|E|^2}{6}$ pairs of distinct points of $E$. This completes the proof of Lemma \ref{alabeck}.
\end{proof}

\medskip

To complete the proof of Theorem \ref{beckfinitefield}, let $l\in \mathcal L$, the set ${\mathcal L}$ provided by Lemma \ref{alabeck}, and $\nu(l)$ is the number of points of $E$ on $l$. It follows from Lemma \ref{alabeck} that
$$
\sum_{l\in \mathcal L} \nu^2(l) > 2 \sum_{l\in \mathcal L} \frac{\nu(l)(\nu(l)-1)}{2} \geq \frac{|E|^2}{2}.
$$
Dividing this by the maximum value of $\nu(l)\leq 2q^{-1}|E|$, we obtain
for the total number of incidences
\begin{equation}
I(E,\mathcal L)=\sum_{l\in \mathcal L} \nu(l)\geq \frac{q|E|}{4},
\label{ibd}\end{equation}
dividing by $2q^{-1}|E|$ once more yields the desired bound
$$|{\mathcal L}|\geq \frac{q^2}{8}.$$

Returning to (\ref{ibd}), by the pigeonhole principle, there exists a point $z\in E$ with at least $\frac{q}{4}$ lines of ${\mathcal L}$ passing through it. This completes the proof of Theorem \ref{beckfinitefield}.
\end{proof}

\subsection{Combining Fourier (section \ref{fouriersubsection}) and Beck (section \ref{becksubsection}) estimates}

\vskip.125in
We now prove the claim ii) of Theorem \ref{2d}.

\begin{proof} Let us refine the initial set $E$ to the set $E'$, containing $z$, provided by Theorem \ref{beckfinitefield}, and exactly $2q^{-1}|E|$ points of $E$ other than $z$ on exactly $\frac{q}{4}$ lines incident to $z$. The flexibility in the choice of constants in the proof of Theorem \ref{beckfinitefield} enables one to treat $\frac{q}{4}$ as integer. It follows that

\begin{equation}|E'|>\frac{|E|}{8}\geq 8q \log_2 q.\label{eprlb}\end{equation}
Now place $z$ to the origin: let $E'_z=E'-z$ an apply Theorem \ref{L2}. More precisely, in the application of the Lemma let $G={E'_z}^{\perp}$ and $F=E'_z\setminus\{0\}.$
Apply the estimate (\ref{L2estimate}) to the sets $F$ and $G$. In view of (\ref{eprlb}) and by Lemma \ref{alabeck}, the quantity $|F\cap l_x|$, for any $x\neq 0$, in the second term of the estimate (\ref{L2estimate}) is bounded by $\frac{16|F|}{q}$. Thus, once more by (\ref{eprlb}), the second term in the estimate (\ref{L2estimate}) is dominated by the first one. Indeed, it follows by (\ref{eprlb}) and the construction of $|E'_z|$ that the number of pairs $(x,y)\in E'_z\times E'_z$, such that $x,y$ lie on the same line through the origin, which is $O(q^{-1}|E|)$, is $o(q^{-2}|E'_z|^2)$. Then
the estimate (\ref{L2estimate}) yields
\begin{equation}
\sum_t \nu^2(t) \leq \frac{|E'|^4}{q}(1+o(1)).
\label{alm}\end{equation}

Now the claim ii) of Theorem (\ref{2d}) follows by the Cauchy-Schwarz inequality in the following way. We have
$$\begin{aligned}|V^z_2(E)|&\geq
(|V^z_2(E')|=|\{t:\nu(t)\neq 0\}|)\\ &  \geq \frac{|\#(x,y)\in F\times G: x\cdot y\neq 0\}|^2}{\sum_t \nu^2(t) }\\
&\geq \frac{|E'|^4(1-o(1))}{ q^{-1}|E'|^4(1+o(1))}\\
&>\frac{q}{2}.\end{aligned}
$$
This completes the proof of part ii) of Theorem \ref{2d}.

\end{proof}

\vskip.125in

\section{Proof of Corollary \ref{dveruki}}

\begin{proof} We proceed by induction on the dimension. Suppose that part i) holds for the dimension $d-1$, $d \ge 3$ and part ii) holds for dimension $d-1$ with $d \ge 4$. The base of the induction for part i) is Theorem \ref{2d}, and for part ii) it is the result of Vinh (\cite{V2}) that $V_3(E)=\F_p\setminus\{0\}$ for $E\subseteq\F_q^3$ with $|E|\geq 2q^2$.

Consider now the intersections of $E\subseteq \F_q^d$ with hyperplanes $H_c=\{x: x_d= c\}$. By the pigeonhole principle, there exists $c$ such that $|E \cap H_c |>q^{d-2}$ for part i), and $|E \cap H_c| \geq 2q^{d-2}$ for part ii) of Corollary \ref{dveruki}.

Since $V_d(E)$ is invariant under translations, we may assume that $c=0$.  In addition, since $|E \cap H_c|>q^{d-1}$, there must be a point $z\in E \backslash H_0$, which means that $z_d \not=0$.

By the induction assumption the set $V_{d-1}(E \cap H_0)$ satisfies the conclusion of Corollary \ref{dveruki}. It follows that
$$
|V_d(E)|\geq |V_d^z((E \cap H_0) \cup \{z\})| = |V_{d-1}(E \cap H_0)|.
$$

To verify the latter equality, since $x^j_d=0$ for every $x^j \in E \cap H_0$, the elements of $V_d^z(E)$ are determinants of size $d+1$ of the form
$$
\left| \begin{array}{ccccc} 1&\ldots&1&1\\
x_1^1&\ldots&x^d_1&z_1\\
\vdots&\ddots&\ddots&\vdots\\
x^1_{d-1}&\ldots&x^d_{d-1}&z_{d-1}\\
0&\ldots&0&z_d
\end{array} \right| = z_d \left| \begin{array}{ccccc} 1&\ldots&1\\
x_1^1&\ldots&x^d_1\\
\vdots&\ddots&\vdots\\
x^1_{d-1}&\ldots&x^d_{d-1}.
\end{array} \right|
$$
This completes the proof of Corollary \ref{dveruki}.
\end{proof}

\end{document}